\documentclass[12pt,epsfig,amsfonts]{amsart} 
\usepackage{amsmath,amsthm,amssymb,amscd,epsfig,color}
\usepackage{graphicx}
\usepackage{mathrsfs}

\newcommand{\1}{\mbox{1}\hspace{-0.25em}\mbox{l}}

\newtheorem*{theorema}{Theorem A}
\newtheorem*{theoremb}{Theorem B}
\newtheorem*{theoremc}{Theorem C}
\newtheorem{prop}{Proposition}[section]
\newtheorem{lemma}[prop]{Lemma}

\newtheorem*{cor1}{Corollary}

\theoremstyle{definition}

\numberwithin{equation}{section}

\begin{document}

\author{Hiroki Takahasi}


\address{Department of Mathematics,
Keio University, Yokohama,
223-8522, JAPAN} 
\email{hiroki@math.keio.ac.jp}
\urladdr{\texttt{http://www.math.keio.ac.jp/~hiroki/}}


\title[Large deviation principles
 for countable Markov shifts]{Large deviation principles\\ for
 countable Markov shifts} 
 
 \date{\today}
 
 \maketitle

  \begin{abstract}
 We establish the large deviation principle for a topological Markov shift
 over infinite alphabet
which satisfies strong combinatorial assumptions 
called ``finite irreducibility'' or ``finite primitiveness''. 
More precisely, we assume the existence of
a Gibbs state for a potential $\phi$ in the sense of Bowen, and prove the level-2 Large Deviation Principles for
the distribution of empirical means under the Gibbs state, as well as that of weighted periodic points and iterated pre-images.
The rate function is written with the pressure and the free energy associated with the potential $\phi$.
  \end{abstract}



 \section{Introduction}
 The theory of large deviations is concerned with the concentration of probability measures.
A sequence $\{\mu_n\}_{n=1}^\infty$ of Borel probability measures on 
a topological space $\mathcal X$ satisfies
{\it the Large Deviation Principle} (LDP) if 
there exists a lower semi-continuous function $I\colon\mathcal X\to[0,\infty]$
 which satisfies the following:
\begin{itemize}
\item [-]{\rm (lower bound)} for every open set $\mathcal G\subset\mathcal X$,
\begin{equation}\label{low}
\liminf_{n\to\infty}\frac{1}{n} \log \mu_n(\mathcal G)\geq-\inf_{\mathcal G}I;\end{equation} 

\item[-] {\rm (upper bound)} for every closed set $\mathcal K\subset\mathcal X$,
\begin{equation}\label{up}
\limsup_{n\to\infty}\frac{1}{n} \log \mu_n(\mathcal K)\leq
-\inf_{\mathcal K}I,\end{equation}

\end{itemize} 
where $ \log 0=-\infty$ and $\inf\emptyset=\infty$. 
The function $I$ is called {\it a rate function}.
It is called {\it a good rate function} if
 the set $\{x\in\mathcal X\colon I(x)\leq\alpha\}$ is compact
  for every $\alpha\geq0$.
The last definition makes sense only if $\mathcal X$ is non-compact.

A number of important transformations with arithmetic or geometric origin
are modeled by symbolic dynamical systems over
infinite alphabet. 
The aim of this paper is to establish the LDP for a class of such symbolic systems,
with a view to providing refined descriptions of the original dynamics.

We introduce our setting in more precise terms.
   Let $S$ be a countable set and denote by $\mathbb N$ the set of non-negative integers.
  Denote by $S^{\mathbb N}$
  the set of all one-sided infinite sequences over $S$ endowed
with the product topology of the discrete topology on $S$, namely
 $$S^{\mathbb N}=\{x=(x_0,x_1,\ldots)\colon x_i\in S,\  i\in\mathbb N\}.$$
  The left shift $\sigma$ acts continuously on $S^{\mathbb N}$ by $(\sigma x)_i=x_{i+1}$.
   Let $A=(A_{ij})_{S\times S}$ be a matrix of zeros and ones with no column or row which is all made of zeros.
 {\it A} (one-sided) {\it topological Markov shift} $X$ generated by $A$ is defined by
$$X=\{x\in S^{\mathbb N}\colon A_{x_ix_{i+1}}=1 \text{ for every }i\in\mathbb N\}.$$
  If $S$ is a countably infinite (resp. finite) set, we call $X$ {\it a countable (resp. finite) Markov shift}.
  If all entries of $A$ are $1$, $X$ is called {\it a full shift}.
   The restriction of $\sigma$ to $X$ is still denoted by $\sigma$.
  For an $n$-string $w$ of elements of $S$,
  denote $|w|=n$.
 For two strings $v=v_0\cdots v_{m-1}$, $w=w_0\cdots  w_{n-1}$ of elements of $S$
 denote by $vw$ the concatenated string
$v_0\cdots v_{m-1}w_0\cdots w_{n-1}$ which is of length $|v|+|w|$.
This notation extends in the obvious way to concatenations of arbitrary finite number of strings.
An $n$-string $w_0w_1\cdots w_{n-1}$ 
  is {\it admissible} if $n=1$, or else $n\geq2$ and $A_{w_{i}w_{i+1}}=1$ holds for $i=0,1,\ldots,n-1$.
 Denote by $E^n$ the set of admissible strings of length $n$ and put $E^*=\bigcup_{n=1}^\infty E^n$.
 For convenience, put $E^0=\emptyset$ and
  $|w|=0$, $vw=v=wv$
 for $w\in E^0$ $v\in E^*$.

 For each $w=w_0\cdots w_{n-1}\in E^n$ 
define {\it an $n$-cylinder} by
 $$[w]=[w_0,\ldots,w_{n-1}]=\{x\in  X\colon x_i=w_i
 \text{ for }i=0,\ldots,n-1\}.$$
For a subset $W$ of $E^*$ write $[W]=\bigcup_{w\in W}[w]$.
 For each $(a,b)\in S\times S$ and an integer $n>1$ define
$$E^n(a,b)=\{w\in E^n\colon w_0=a\text{ and }w_{n-1}=b\},$$
and 
$$E^n(a)=\bigcup_{b\in S}E^n(a,b).$$

 Let $\phi\colon X\to\mathbb R$ be a function.
 For an integer $n\geq1$ write
$ S_n\phi=\sum_{i=0}^{n-1}\phi\circ \sigma^i$, and put
$S_0\phi=0$ for convenience.
 A Borel probability measure $\mu_\phi$ on $X$ is {\it a Gibbs state (in the sense of Bowen) for the potential $\phi$}
(cf. \cite{Bow75,MauUrb01,Rue78,Sar99})
if there exist constants $c_0\geq1$ and $P\in\mathbb R$ such that
 for every $n\geq1$
 and every $x=(x_i)_{i\in\mathbb N}\in X$,
 \begin{equation}\label{Gibbs}
 c_0^{-1}\leq\frac{\mu_\phi  [x_0,\ldots,x_{n-1}] }{\exp\left(-Pn+S_n\phi(x)\right)}\leq c_0.
 \end{equation}
 
It is now classical \cite{Bow75,Rue78} that for a topologically mixing, finite 
Markov shift and a H\"older continuous potential $\phi$ there exists a unique
 $\sigma$-invariant Gibbs state, and it coincides with the unique equilibrium state for $\phi$ characterized by the variational principle,
and the constant $P$ in \eqref{Gibbs} equals the topological pressure of $\phi$.
The construction of $\sigma$-invariant Gibbs states for countable Markov shifts
was carried out by Sarig \cite{Sar99,Sar03}, and 
by Mauldin and Urba\'nski \cite{MauUrb01}
under weaker assumptions on transition matrices and stronger\footnote{
In \cite{MauUrb01}, continuity assumptions on $\phi$ stronger than \cite{Sar03}  were used,
but the proof in \cite{MauUrb01} works verbatim in the context of \cite{Sar03}.}  assumptions on potentials
than \cite{Sar99,Sar03}.
Our presentation of main results 
follows \cite{MauUrb01} in order to work with the weakest possible assumptions
on both transition matrices and potentials.
We assume the existence of a Gibbs state, and establish the LDP for several interesting sequences of measures.
Note that \eqref{Gibbs} differs from the definition of the Gibbs state 
in statistical mechanics \cite{Rue78}.
For the LDPs with respect to the Gibbs state with shift-invariant absolutely summable interactions,
see \cite{Bry92,Com86,EKW94,FolOre88,Oll88} and \cite[Theorem 8.6]{RasSep}.

 Denote by $\mathcal M$ the space of Borel probability measures on $X$
 endowed with the weak*-topology.
We establish the (level-2) LDP 
for the following three sequences of 
 Borel probability measures on $\mathcal M$: 
  
  \begin{itemize}
 
\item[]\noindent{1. (Empirical means).}  
For each $x\in X$ and an integer $n\geq1$ define 
$$\delta_x^n=\frac{1}{n}\sum_{i=0}^{n-1}\delta_{\sigma^ix},$$
  with $\delta_{\sigma^ix}$ the unit point mass at $\sigma^ix$.
Denote by $\xi_n$ the distribution of the $\mathcal M$-valued random variable 
 $x\mapsto\delta_x^n$ on the probability space $(X,\mu_\phi)$;

\item[]\noindent{2. (Weighted periodic points).} 
Let $A$ be a countable subset of $X$.
For each integer $n\geq1$ put $$Z_n(\phi,A)=\sum_{x\in A}\exp S_n\phi(x).$$
Define
$$\eta_n=\frac{1}{Z_n(\phi,{\rm Per}_n(\sigma))}
\sum_{x\in{\rm Per}_n(\sigma)} \exp S_n\phi(x)\delta_{\delta_x^n},$$
with ${\rm Per}_n(\sigma)=\{x\in X\colon\sigma^nx=x\}$ and
$\delta_{\delta_x^n}$ the unit point mass at $\delta_x^n$;

\item[]{3. (Weighted iterated pre-images).} 
Fix $y\in X$ and define
$$\zeta_{y,n}=
 \frac{1}{Z_n(\phi,\sigma^{-n}y)}\sum_{x\in\sigma^{-n}y} \exp S_n\phi(x)\delta_{\delta_x^n},$$
with $\sigma^{-n}y=\{x\in X\colon\sigma^nx=y\}$.
\end{itemize}


 For each $\sigma$-invariant measure $\mu\in\mathcal M$,
 denote by $h(\mu)$ the Kolmogorov-Sina{\u\i} entropy of $\mu$ with respect to $\sigma$.
 It is possible that $h(\mu)=\infty$ if $\#S=\infty$.
 If $\sup\phi<\infty$ then
define
$$\mathcal M_\phi(\sigma)=\left\{\mu\in\mathcal M\colon\text{$\mu$ is $\sigma$-invariant and  $\int\phi d\mu>-\infty$}\right\}.$$
 The condition $\sup\phi<\infty$ guarantees that $\int\phi d\mu$ is well-defined for every $\mu\in\mathcal M$,
 though possibly $\int\phi d\mu=-\infty.$

A countable Markov shift $X$ is {\it finitely irreducible}
if there exists a finite set $\Lambda\subset E^*$ 
such that for all $i,j\in E^*$ there exists $\lambda\in\Lambda$
for which $i\lambda j\in E^*$.
If $X$ is finitely irreducible and the finite set $\Lambda$
consists of strings of the same length $N$, then $X$ is called {\it finitely primitive}.
Notice that the set $\Lambda$ associated either with a finitely irreducible or
primitive matrix can be taken to be empty for the full shift
$X=S^\mathbb N$ (in which case $N=0$).
The finite primitiveness implies that the shift map is topologically mixing.

The construction of (shift-invariant) Gibbs states in \cite{MauUrb01,MauUrb03} assumes
the finite irreducibility or primitiveness, and we also require these conditions.
In the case $X$ is topologically mixing and $\phi$ has summable variations, the finite primitiveness is a necessary condition
for the existence of a shift-invariant Gibbs state \cite{Sar03}.


\begin{theorema}
Let $X$ be a finitely irreducible countable Markov shift, $\phi\colon X\to\mathbb R$ a
measurable function and $\mu_\phi$ a Gibbs state for the potential $\phi$. 
Then $\{\xi_n\}$ is exponentially tight and
satisfies the LDP with the convex
 good rate function $I$ 
 given by
\begin{equation*}\label{rate}
 I(\mu)
=
-\inf_{\mathcal G \ni \mu}\sup_{\mathcal G}F.
\end{equation*}
The infimum is taken over all open sets $\mathcal G\subset\mathcal M$ containing $\mu$,
and $F \colon \mathcal M \to [-\infty, 0]$ is defined by
\begin{equation}\label{ef}
\begin{split} F(\nu)
=
\begin{cases}-P+h(\nu)+\int\phi d\nu &\text{ if $\nu\in\mathcal M_\phi(\sigma)$};\\ 
-\infty&\text{ otherwise.}\end{cases}\end{split}\end{equation}
\end{theorema}

To show the LDP in this non-compact setting, it is necessary to control escapes of probability masses to infinity.
The exponential tightness (see Proposition \ref{tight} for the definition)
asserts that masses are concentrated on compact sets, at least on an exponential scale.
This property is used to treat non-compact closed sets.

If $X$ is the full shift and 
 $m$ is a probability measure on $S$ such that
 $m[a]>0$ holds for every $a\in S$,
then the product measure $m^{\otimes \mathbb N}$ 
is the unique shift-invariant Gibbs state for the potential  $\phi(x)=-\log m[x_0].$
The sequence of $\mathcal M$-valued random variables $x\mapsto \delta_{\sigma^nx}$ $(n=1,2,\ldots)$ on $(X, m^{\otimes \mathbb N})$ are independent
and identically distributed, and 
the LDP for the corresponding $\{\xi_n\}$ is known as Sanov's theorem. 
Theorem A 
allows for the lack of independence introduced by the
potential $\phi$.

The finite primitiveness can be used to find periodic points of the same periods and iterated pre-images
of the same lengths. We obtain the following result.

\begin{theoremb}
Let $X$ be a finitely primitive countable Markov shift, $\phi\colon X\to\mathbb R$ a
measurable function and $\mu_\phi$ a Gibbs state for the potential $\phi$. 
The  $\{\eta_n\}$ and $\{\zeta_{y,n}\}$ $(y\in X)$ are exponentially tight and
satisfy the LDP with the same rate function as in Theorem A.
\end{theoremb}

Theorems A and B extend
the results of Takahashi \cite{Tak84,Tak87} and Kifer \cite{Kif90,Kif94} for 
finite Markov shifts to countable ones.
In \cite{Tak84,Tak87}, Takahashi treated the distribution of empirical means under the Gibbs state.
In \cite{Kif90}, Kifer provided a unified functional analytic approach 
 to establishing the LDP which is in particular 
applicable to finite Markov shifts.
 In \cite{Kif94} he also obtained the LDP
for the distribution of periodic points. 
 Orey and Pelikan \cite{OrePel89} proved the LDP for uniformly hyperbolic systems (Anosov diffeomorphisms),
which via Markov partitions can be modeled by finite Markov shifts. 
 The rate functions in these settings are given
by the difference between the pressure and the free energy, while
 in Theorems A and B it is not possible to take $-F$ as a rate function.
For instance, if $X$ is the full shift then $-F$ is not lower semi-continuous
  (see the remark at the end of this paper). 

It is also relevant to put an initial condition and consider the LDP.
For each $a\in S$, $y\in X$ and an integer $n\geq1$ define
 $$\eta_{a,n}=\frac{1}{Z_{n}(\phi,[a]\cap{\rm Per}_n(\sigma))}\sum_{x\in[a]\cap{\rm Per}_n(\sigma)} \exp S_n\phi(x)\delta_{\delta_x^n};$$
  $$\zeta_{a,y,n}=\frac{1}{Z_{n}(\phi,[a]\cap\sigma^{-n}y)}\sum_{x\in[a]\cap\sigma^{-n}y} \exp S_n\phi(x)\delta_{\delta_x^n}.$$
The latter distributions
define a thermodynamic limit with boundary condition $y$, conditioned on $[a]$,
see \cite{Sar15} for details.

\begin{theoremc}
Let $X$ be a finitely primitive countable Markov shift, $\phi\colon X\to\mathbb R$ a
measurable function and $\mu_\phi$ a Gibbs state for the potential $\phi$. 
For every $a\in S$ and $y\in X$,
the $\{\eta_{a,n}\}$ and $\{\zeta_{a,y,n}\}$ are exponentially tight
and
satisfy the LDP with the same rate function as in Theorem A.
\end{theoremc}

 Known results on large deviations for countable Markov shifts are very much limited.
Under additional assumptions on $(X,\phi)$ and a bounded function $\varphi\colon X\to\mathbb R$,
the following local large deviations for $\{\zeta_{a,y,n}\}$ is a consequence of \cite[Theorem 7.4]{Sar15}:
for any $\epsilon>0$ there exists $c_\varphi(\epsilon)>0$ such that
$$\frac{1}{Z_{n}(\phi,[a]\cap\sigma^{-n}y)}\sum_{\stackrel{x\in[a]\cap\sigma^{-n}y}{|(1/n)S_n\varphi(x)-\int\varphi d\mu_\phi|>\epsilon}}\exp S_n\phi(x)
\leq\exp\left(- c_\varphi(\epsilon)n\right),$$
for all $n$ large enough. Similar exponential 
bounds were obtained in \cite[Theorem 3.5]{Yur05} under other strong combinatorial assumptions on $X$.
These results indeed
provide exponential bounds on small fluctuations near the mean $\int\varphi d\mu_\phi$ (for small $\epsilon$), but
 do not provide enough information for large $\epsilon$,
and do not imply the LDP. The only one result on the LDP
for countable Markov shifts we are currently aware of is due to Denker and Kabluchko 
\cite[Theorem 3.3]{DenKab07}, who showed the level-1 LDP
for Gibbs-Markov maps and for a certain class of bounded observables.
In the context of smooth dynamical systems, local large deviations results
were obtained for non-uniformly hyperbolic systems admitting
inducing schemes with countably infinite number of branches \cite{MelNic08,ReyYou08}.
Although part of arguments in \cite{MelNic08,ReyYou08} may be applicable to our setting,
they will yield only local large deviations results too, not the LDP.


One important consequence of Theorem A is Varadhan's abstraction of Laplace's method.
Denote by $C(X)$ the space of $\mathbb R$-valued bounded continuous functions on $X$
endowed with the supremum norm.
Under the hypotheses of Theorem A,
for each $\varphi\in C(X)$ the limit
$$Q(\varphi)=\lim_{n\to\infty}\frac{1}{n}\log\int \exp{S_n\varphi}d\mu_\phi$$
exists and satisfies
\begin{equation*}
Q(\varphi)=\sup_{\mu\in\mathcal M}\left(\int\varphi d\mu-I(\mu)\right),\end{equation*}
as shown in \cite{Var84}. By convex duality, this implies
$$I(\mu)=\sup_{\varphi\in C(X)}\left(\int\varphi d\mu-Q(\varphi)\right)\quad\text{for every $\mu\in\mathcal M$.}$$
This follows, e.g., from \cite[Lemma 4.5.8]{DemZei98} if we use the natural
embedding of $\mathcal M$ into the topological vector space of signed measures on $X$.

Another important consequence of Theorems A and B is the level-1 LDP.
Let $d\geq1$ be an integer and $\varphi_1,\ldots,\varphi_d\in C(X)$.
By the contraction principle, 
the sequence of 
 distributions of $\mathbb R^d$-valued random variables 
 $x\mapsto ((1/n)S_n\varphi_1,\ldots,(1/n)S_n\varphi_d)$
 satisfies the LDP, with the convex good rate function
 $I_\varphi\colon\mathbb R^d\to[0,\infty]$ given by
  $$I_\varphi(\alpha_1,\ldots,\alpha_d)= 
\inf\left\{I(\mu)\colon\mu\in\mathcal M, \left(\int\varphi_1 d\mu,\ldots,\int\varphi_d d\mu\right)=(\alpha_1,\ldots,\alpha_d)\right\},$$
which is finite
 if and only if $\alpha_j\in\left[\inf_{\mu\in\mathcal M_\phi(\sigma)}\int\varphi_jd\mu,
\sup_{\mu\in\mathcal M_\phi(\sigma)}\int\varphi_jd\mu\right]$ holds for $j=1,\ldots,d$.
The case $d=1$ extends the result of  Denker and Kabluchko \cite[Theorem 3.3]{DenKab07}, in which
the level-1 LDP was shown for a limited class of functions including those which
 depend only on the first finite number of symbols.



We illustrate our results with the regular continued fraction expansion
$$x=\cfrac{1}{a_1(x)+\cfrac{1}{a_2(x)+\cdots}},$$
where $x\in(0,1)\setminus\mathbb Q$ and each digit $a_i(x)$ $(i=1,2,\ldots)$ is a positive integer.
We investigate frequencies with which a given integer $k$ appears in this expansion.
The digits are generated by
iterating the Gauss transformation $T\colon (0,1]\to[0,1)$ given by 
$Tx=1/x-\lfloor1/x\rfloor$, namely
$a_i(x)=k$ if and only if $T^{i-1}x\in(\frac{1}{k+1},\frac{1}{k})$.
Denote by ${\rm Leb}$ the restriction of the Lebesgue measure to $(0,1)$.
 The Gauss measure
$\frac{1}{\log2}\frac{dx}{1+x}$ is 
the unique $T$-invariant Borel probability measure that is absolutely continuous with
respect to {\rm Leb}. 
For each integer $n\geq1$ define a counting function $F_{k,n}\colon(0,1)\setminus\mathbb Q\to\mathbb N$ by
 $$F_{k,n}(x)=\#\{1\leq i\leq n\colon a_i(x)=k\}.$$
Since the Gauss measure is ergodic, Birkhoff's Theorem gives 
$$\frac{1}{n}F_{k,n}\to\frac{1}{\log2}\log\frac{(k+1)^2}{k(k+2)}\ \text{
$(n\to\infty)$ {\rm Leb}-a.e.}$$
Following orbits of $T$ over the Markov partition $\{(\frac{1}{k+1},\frac{1}{k}]\}_{k=1}^\infty$
one can model $T$ by the countable full shift.
Denoted by $\pi\colon S^{\mathbb N}\to(0,1)$ the conjugacy $T\circ\pi=\pi\circ\sigma$.
The Gibbs state for the potential 
$\phi=-\log|DT\circ\pi|$ corresponds to
the Gauss measure.
From \cite[Proposition 3.4]{DenKab07}, the minimizer of the level-1 rate function associated
with the indicator function $\1_{(\frac{1}{k+1},\frac{1}{k})}$ of the interval $(\frac{1}{k+1},\frac{1}{k})$ is unique.
Hence, the corresponding level-1 LDP reads as follows.
For comparison, see \cite[Theorem 3.3, Proposition 3.4]{DenKab07}.

\begin{cor1}\label{apply}
For every integer $k\geq1$ 
and every
$y\in \pi(S^{\mathbb N})$ the following holds:
\begin{itemize} 
\item[(a)] for every $\alpha\in\left(\frac{1}{\log2}\log\frac{(k+1)^2}{k(k+2)},1\right]$,
\begin{align*}
\lim_{n\to\infty}&\frac{1}{n}\log{\rm Leb} \left\{x\in(0,1)\setminus\mathbb Q\colon \frac{1}{n}F_{k,n}(x)\geq\alpha\right\}\\
&= \lim_{n\to\infty}\frac{1}{n}\log 
\sum_{\stackrel{x\in(0,1)\setminus\mathbb Q}{T^nx=x,\ 
(1/n)F_{k,n}(x)\geq\alpha }} 
|DT^n(x)|^{-1}\\
\ &= \lim_{n\to\infty}\frac{1}{n}\log 
\sum_{\stackrel{x\in(0,1)\setminus\mathbb Q}{T^nx=y,\ (1/n)F_{k,n}(x)\geq\alpha }} 
|DT^n(x)|^{-1}\\
\ &=-\inf\left\{I(\mu)\colon\mu\in\mathcal M,\int \text{\rm $\1_{(\frac{1}{k+1},\frac{1}{k})}$}d(\mu\circ\pi^{-1})=\alpha\right\}
\in(-\infty,0);
\end{align*}
\item[(b)] for every $\alpha\in\left[0,\frac{1}{\log2}\log\frac{(k+1)^2}{k(k+2)}\right)$,
 \begin{align*}
\lim_{n\to\infty}&\frac{1}{n}\log{\rm Leb} \left\{x\in(0,1)\setminus\mathbb Q\colon \frac{1}{n}F_{k,n}(x)\leq\alpha\right\}\\
&= \lim_{n\to\infty}\frac{1}{n}\log 
\sum_{\stackrel{x\in(0,1)\setminus\mathbb Q}{T^nx=x,\ 
(1/n)F_{k,n}(x)\leq\alpha }} 
|DT^n(x)|^{-1}\\
\ &= \lim_{n\to\infty}\frac{1}{n}\log 
\sum_{\stackrel{x\in(0,1)\setminus\mathbb Q}{T^nx=y,\ (1/n)F_{k,n}(x)\leq\alpha }} 
|DT^n(x)|^{-1}\\
\ &=-\inf\left\{I(\mu)\colon\mu\in\mathcal M,\int \text{\rm $\1_{(\frac{1}{k+1},\frac{1}{k})}$}d(\mu\circ\pi^{-1})=\alpha\right\}
\in(-\infty,0);
\end{align*}
\item[(c)]
the three limits below exist for all $\beta\in\mathbb R$, differentiable at $\beta=0$ and 
\begin{align*}\frac{1}{\log2}\log\frac{(k+1)^2}{k(k+2)}&=\left.\frac{d}{d\beta}\right|_{\beta=0}
\lim_{n\to\infty}\frac{1}{n}\log\int
\exp(\beta F_{k,n}(x))dx\\
&=\left.\frac{d}{d\beta}\right|_{\beta=0}\lim_{n\to\infty}\frac{1}{n}\log\left(\sum_{\stackrel{x\in (0,1)\setminus\mathbb Q}{T^nx=x}} 
\exp(\beta F_{k,n}(x))|DT^n(x)|^{-1}\right)\\
&=\left.\frac{d}{d\beta}\right|_{\beta=0}\lim_{n\to\infty}\frac{1}{n}\left(\log\sum_{\stackrel{x\in (0,1)\setminus\mathbb Q}{T^nx=y}} 
\exp(\beta F_{k,n}(x))|DT^n(x)|^{-1}\right).
\end{align*}\end{itemize}
\end{cor1}

Since the pressure is $0$, the contributions from the normalizing factors 
 in the formulas in (a) and (b) disappear as $n\to\infty$ 
(see Proposition \ref{pressure}).
Item (c) is a consequence of the general theory on large deviations
 \cite[Theorem II. 6.3]{Ell85}.

The rest of this paper consists of three sections entirely dedicated to proofs of the theorems.
   After a few preliminaries in
  $\S$2 we prove the lower bound \eqref{low} for all open sets in $\S$3,
  and then the upper bound \eqref{up} for all closed sets in $\S$4.
Our argument is a dynamical one as briefly outlined below, 
 inspired by that of Takahashi \cite{Tak84,Tak87}. 
 New ingredients are necessary for handling difficulties arising from 
 the non-compactness of $X$ and $\mathcal M$.

A useful property for a proof of the large deviations lower bound for all open sets
 is {\it the entropy-density of ergodic measures} (see \S \ref{specify} for the definition).
This property permits the reduction of the proof of the lower bound 
to the case where the measure in consideration is ergodic
\cite{EKW94,FolOre88,You90}. 
The entropy-density in our setting was shown in \cite[Main Theorem]{Tak}, and
its slight variant taking the unboundedness of the potential $\phi$ into consideration 
(see Lemma \ref{specification}) suffices to perform this reduction.
Estimates for ergodic measures are carried out by combining the Gibbs property
and an approximation of ergodic measures with a finite number
 of cylinders (separated sets), which is well-known for compact metric spaces \cite{EKW94}
  and is still valid in our setting of infinite alphabet
 (see Lemma \ref{katok''}).
    
 

The exponential tightness allows us to reduce 
the proof of the upper bound for all closed sets
to that for all compact sets.
To show this property,
we modify a portion of a proof of Sanov's theorem 
on the LDP for the distribution 
of empirical means associated with i.i.d. random variables.
The lack of independence in our setting is compensated by
a bounded distortion property of the Gibbs state (see Lemma \ref{distor}).
The finite irreducibility is used in a crucial way to treat all compact sets.
We construct a finite number of finite subsystems (finite full shifts) and
invariant probability measures on each,
and use them altogether to deduce the desired upper bound.

\section{Preliminaries}
For the rest of this paper we assume $S=\mathbb N$ for simplicity,
and $X$ always denotes a countable Markov shift.
In this section we collect and prove a few preliminary results which will be frequently used later.

\subsection{Mild distortions}\label{weaktop}

The topology on $X$ is metrizable by a metric
 $d(x,y)=\exp\left(-\inf\{i\in\mathbb N\colon x_i\neq y_i\}\right)$
 with the convention $\exp(-\infty)=0$.
Denote by $C_u(X)$ the set of uniformly continuous elements of $C(X)$.
For a function $\varphi\colon X\to\mathbb R$ and an integer $n\geq1$ define
  $$D_n(\varphi)=\sup_{w\in E^n}\sup_{x,y\in[w]}S_{n}\varphi(x)-S_{n}\varphi(y).$$
  Notice that $D_n(\varphi)\leq D_1(\varphi)n$ holds for every $n\geq1$.
  The regularity of functions needed in most of our argument is $D_n(\varphi)=o(n)$,
  which is satisfied for elements of $C^u(X)$.

\begin{lemma}\label{varlem}{\rm (\cite[Proposition 6.2(b)]{FieFieYur02}).}
If $\varphi\in C_u(X)$, then
$D_n(\varphi)=o(n)$ $(n\to\infty)$.
\end{lemma}

Each $\varphi\in C(X)$ defines a functional $\mu\in\mathcal M\mapsto\int\varphi d\mu$.
The weak*-topology is
the coarsest topology on $\mathcal M$ 
which makes every functional $\varphi(\cdot)$, $\varphi\in C(X)$  continuous. 
As $X$ is a Polish space,
the weak*-topology is metrizable and 
 $\mathcal M$ becomes a Polish space.
The weak*-topology coincides with
 the coarsest topology which makes every $\varphi(\cdot)$, $\varphi\in C_u(X)$  continuous (see e.g., \cite[Chapter 9]{Str11}).

\subsection{Properties of Gibbs states}\label{fund}
The existence of a Gibbs state imposes strong restrictions
 on the corresponding potential. 
   \begin{lemma}\label{impose}
Let $\phi\colon X\to\mathbb R$ be a measurable function and 
assume there exists a Gibbs state for the potential $\phi$.
 Then $\sup\phi<\infty$,  $\inf\phi=-\infty$ and
  $\sup_{n\geq1}D_n(\phi)<\infty.$ 
 \end{lemma}
  \begin{proof}
 Immediate from \eqref{Gibbs}.
  \end{proof}

To compensate the lack of independence of the random variables in question,
we use the next ``bounded distortion property'' of Gibbs states.
\begin{lemma}\label{distor}
Let $\phi\colon X\to\mathbb R$ be a measurable function and
 $\mu_\phi$ a Gibbs state for the potential $\phi$ as in \eqref{Gibbs}.
Then the following holds:

\begin{itemize}
\item[(a)] for all $v$, $w\in E^*$ with
$vw\in E^*$,
$$c_0^{-3}\mu_\phi[w]\leq\frac{\mu_\phi[vw]}{\mu_\phi[v]}\leq c_0^3\mu_\phi[w].$$
\item[(b)] for all $v$, $w_i\in E^*$ $(i\in\mathbb N)$ with
$vw_i\in E^*$ for every $i\in\mathbb N$, 
$$c_0^{-3}
\sum_{i\in\mathbb N}\mu_\phi [w_i]\leq\frac{\sum_{i\in\mathbb N}\mu_\phi[ vw_i]}{\mu_\phi[v]}\leq c_0^3
\sum_{i\in\mathbb N} \mu_\phi[w_i].$$
\end{itemize}
\end{lemma}

\begin{proof}
Let $v$, $w\in E^*$ with
$vw\in E^*$.
From \eqref{Gibbs} the following holds:
$$ c_0^{-1}\exp\left(-P|vw|+\sup_{[vw]}S_{|vw|}\phi\right)\leq\mu_\phi[vw]\leq c_0\exp\left(-P|vw|+\inf_{[vw]}S_{|vw|}\phi\right);$$
\begin{equation*}
c_0^{-1}\exp\left(-P|v|+\sup S_{|v|}\phi\right)\leq\mu_\phi[v]\leq c_0\exp\left(-P|v|
+\inf S_{|v|}\phi\right).\end{equation*}
In addition, 
$S_{|vw|}\phi-S_{|v|}\phi=(S_{|w|}\phi)\circ \sigma^{|v|}$ holds on $[vw]$.
 Since $\sigma^{|v|}[vw]\subset [w]$ we obtain
$$ 
\frac{\mu_\phi[vw]}{\mu_\phi[v]}\geq c_0^{-2}\exp\left(-P|w|+\inf_{[w]}S_{|w|}\phi\right)\geq c_0^{-3}\mu_\phi[w]$$
and
$$\frac{\mu_\phi[vw]}{\mu_\phi[v]}\leq
c_0^2\exp\left(-P|w|+\sup_{[w]} S_{|w|}\phi\right)\leq c_0^3\mu_\phi[w].$$
Item (b) is a consequence of (a) and the countable additivity of a measure.
 \end{proof}

\subsection{Expressions of pressure}\label{express}
  Given a measurable function $\phi\colon X\to\mathbb R$ 
  define its {\it pressure} 
$$P(\phi)=\lim_{n\to\infty}\frac{1}{n}\log\sum_{w\in E^n} 
\sup_{[w]}\exp  S_n\phi.$$
As the sequence $n\mapsto \log\sum_{w\in E^n} 
\sup_{[w]}\exp  S_n\phi$ is sub-additive, this limit exists. 
If $\mu_\phi$ is a Gibbs state for the potential $\phi$,
the constant $P$ in \eqref{Gibbs}
is equal to $P(\phi)$, see \cite[Proposition 2.2(a)]{MauUrb01}.

\begin{prop}\label{pressure}
Let $X$ be finitely primitive,
 $\phi\colon X\to\mathbb R$ a measurable function and $\mu_\phi$ a Gibbs state for the potential $\phi$.
Then for every $a\in \mathbb N$ and $y\in X$,
  \begin{align*}
  P(\phi)&= \lim_{n\to\infty}\frac{1}{n}\log Z_n(\phi,[a]\cap{\rm Per}_n(\sigma))=\lim_{n\to\infty}\frac{1}{n}\log Z_n(\phi,{\rm Per}_n(\sigma))\\
  &=\lim_{n\to\infty}\frac{1}{n}\log Z_n(\phi,[a]\cap\sigma^{-n}y)=\lim_{n\to\infty}\frac{1}{n}\log Z_n(\phi,\sigma^{-n}y).\end{align*}
\end{prop}


\begin{proof}
Let $\Lambda\subset E^*$ be the finite set and $N\geq0$ the integer given by the finite primitiveness of $X$.
Recall that $N=0$ if and only if $\Lambda=\emptyset$.
\begin{lemma}\label{transition}
There exists $c>0$ such that
for every $(a,b)\in \mathbb N^2$ and every integer $n>N$,
$$\mu_\phi([a]\cap\sigma^{-n}[b])\geq c\mu_\phi[a]\mu_\phi[b].$$
\end{lemma}

\begin{proof}
Put $c=c_0^{-6}\inf_{\lambda\in\Lambda}\mu_\phi[\lambda]$ if $\Lambda\neq\emptyset$ and $c=c_0^{-6}$ if $\Lambda=\emptyset$.
In the case $\Lambda\neq\emptyset$, for each $w\in  E^{n-N}(a)$ 
fix $\kappa=\kappa(w)\in\Lambda$ with $w\kappa b\in E^{n+1}$.
 Lemma \ref{distor}(a) gives
\begin{equation*}
\frac{\mu_\phi [w\kappa b]}{\mu_\phi [w]}\geq c_0^{-3}\mu_\phi[\kappa b]
\geq  c_0^{-6}\mu_\phi[\kappa]\mu_\phi[b]\geq c\mu_\phi[b].\end{equation*}
Rearranging this inequality and summing the result over all $w \in E^{n-N}(a)$ yields
$$\mu_\phi([a]\cap\sigma^{-n}[b])\geq\sum_{  w\in E^{n-N}(a)}
\mu_\phi[w\kappa b]\geq c\sum_{  w\in E^{n-N}(a)}
\mu_\phi[w]\mu_\phi[b]=c\mu_\phi[a]\mu_\phi[b],$$
as required. A proof for the case $\Lambda=\emptyset$ follows from the obvious modification.
\end{proof}
Returning to the proof of Proposition \ref{pressure},
for every $a\in \mathbb N$ and $n>N$ we have
\begin{align*}
c_0e^{-P(\phi)n}Z_n(\phi,[a]\cap{\rm Per}_n(\sigma))&\geq \mu_\phi([a]\cap\sigma^{-n}[a])
\quad\text{by \eqref{Gibbs}}\\
&\geq c\mu_\phi[a]^2\quad\text{by Lemma \ref{transition}}.\end{align*}
 Hence
 \begin{equation}\label{un}P(\phi)\leq\liminf_{n\to\infty}\frac{1}{n}\log Z_n(\phi,[a]\cap{\rm Per}_n(\sigma)).\end{equation}
On the other hand, \eqref{Gibbs} also implies
\begin{equation*}
c_0^{-1}e^{-P(\phi)n}Z_n(\phi,[a]\cap{\rm Per}_n(\sigma))\leq \mu_\phi[E^n(a)].\end{equation*}
Summing this inequality over all $a\in \mathbb N$ gives
$$c_0^{-1}e^{-P(\phi)n}Z_n(\phi,{\rm Per}_n(\sigma))\leq \sum_{a\in \mathbb N}
\mu_\phi[E^n(a)]=1,$$
and therefore
\begin{equation}\label{dois}\limsup_{n\to\infty}\frac{1}{n}\log Z_n(\phi,{\rm Per}_n(\sigma))\leq P(\phi).\end{equation}
The inequalities \eqref{un} and \eqref{dois} imply the two equalities in the first line in Proposition \ref{pressure}.

Let $y=(y_i)_{i\in \mathbb N}\in X$. 
For every $a\in \mathbb N$ and $n>N$,
\begin{align*}
c_0e^{-P(\phi)n}Z_n(\phi,[a]\cap\sigma^{-n}y)&\geq 
\mu_\phi([a]\cap\sigma^{-n}[y_0])\quad\text{by \eqref{Gibbs}}\\
&\geq c\mu_\phi[a]\mu_\phi[y_0]\quad\text{by Lemma \ref{transition}}.\end{align*}
Hence
\begin{equation}\label{tres}P(\phi)\leq\liminf_{n\to\infty}\frac{1}{n}\log Z_n(\phi,[a]\cap\sigma^{-n}y).\end{equation}
On the other hand, \eqref{Gibbs} also implies
\begin{equation*}
c_0^{-1}e^{-P(\phi)n}Z_n(\phi,[a]\cap\sigma^{-n}y)\leq 
\mu_\phi[E^n(a)].\end{equation*}
Summing this inequality over all $a\in \mathbb N$ gives
\begin{equation*}
c_0^{-1}e^{-P(\phi)n}Z_n(\phi,\sigma^{-n}y)\leq 1,\end{equation*}
and therefore
\begin{equation}\label{quatro}\limsup_{n\to\infty}\frac{1}{n}\log Z_n(\phi,\sigma^{-n}y)\leq P(\phi).\end{equation}
The inequalities \eqref{tres} and \eqref{quatro} imply the two equalities in the second line in Proposition \ref{pressure}.
\end{proof}


 \section{Large deviations lower bound}
 This section is devoted to the proof of the lower bound \eqref{low} for all
 open sets. 
 In $\S$\ref{intermediate} we prove a lemma which approximates each ergodic measure
 with finite entropy with a finite collection of cylinders.
 In $\S$\ref{specify} we show that the proof of the lower bound can be reduced to the case where the invariant measure in question is ergodic.
  In $\S$\ref{keyest} we prove a key lower bound,
  and from it deduce the desired one in $\S$\ref{endrow}.

\subsection{Approximation of ergodic measures}\label{intermediate}


  The next lemma approximates 
ergodic measures with a finite collection of cylinders
in a particular sense. 
\begin{lemma}\label{katok''}
Let 
$l\geq1$ be an integer and let $\varphi_j\colon X\to\mathbb R$ 
satisfy $\sup\varphi_j<\infty$ and
$D_n(\varphi_j)=o(n)$ for $j=1,\ldots,l$. 
For any $\sigma$-invariant ergodic measure $\mu\in\mathcal M$ with $h(\mu)<\infty$,
$\int\varphi_jd\mu>-\infty$ for $j=1,\ldots,l$
and any $\epsilon>0$ there exist $n_0\geq1$
such that for every integer $n\geq n_0 $
there exists a finite subset $F^{n}$ of $E^{n}$ for which the following holds:
\begin{equation}\tag{a}
\left|\frac{1}{n}\log\#F^n- h(\mu)\right|\leq\epsilon;\end{equation}
\begin{equation}\tag{b}\sup_{[F^n]}\left|\frac{1}{n}S_{n}\varphi_j-\int\varphi_j d\mu\right|\leq\epsilon\quad
\text{for $j=1,\ldots,l.$}\end{equation}
\end{lemma}

\begin{proof}
Let $\epsilon>0$.
For each integer $n\geq1$ denote by $F^n$ the set of $w\in E^n$ 
for which the following holds:
\begin{equation}\label{cond1}
\exp\left(-\left(h(\mu)+\epsilon\right)n\right)\leq\mu[w]\leq \exp\left(-\left(h(\mu)-\frac{\epsilon}{2}\right)n\right);
\end{equation}
 \begin{equation}\label{cond2}
\inf_{[w]}\left|\frac{1}{n}S_{ n}\varphi_j-\int\varphi_j d\mu\right|\leq\frac{\epsilon}{2}\quad\text{for $j=1,\ldots,l$.}
\end{equation}
Since $h(\mu)<\infty$ and the partition $\{[k]\colon k\in\mathbb N\}$ is a generator,
Shannon-McMillan-Breiman's Theorem 
and Birkhoff's Theorem together imply 
$\mu[F^n]\to1$ as $n\to\infty$. For $n$ large enough so that $\mu[F^n]\geq1/2$, \eqref{cond1} implies
$$\frac{1}{2}\exp\left(\left(h(\mu)-\frac{\epsilon}{2}\right)n\right)\leq\#F^n\leq\exp\left((h(\mu)+\epsilon)n\right),$$
which yields (a).
Item (b) follows from \eqref{cond2} provided $n$ is large enough so that
$(1/n)D_{n}(\varphi_j)
\leq\epsilon/2$ holds for $j=1,\ldots,l$.
\end{proof}

\subsection{Reduction to ergodic measures}\label{specify}

We say $X$ is {\it transitive} if for any $a,b\in\mathbb N$ there exists an integer $n\geq1$
such that $[a]\cap\sigma^{-n}[b]\neq\emptyset$. Clearly, the finite irreducibility implies the transitivity.
 For transitive countable Markov shifts,
ergodic measures are {\it entropy-dense} 
\cite[Main Theorem]{Tak}:
for any non-ergodic $\mu$ and $\epsilon>0$ there exists an ergodic $\nu$ which satisfies
$h(\nu)>h(\mu)-\epsilon$.
The proof of \cite[Main Theorem]{Tak} works verbatim to show the next lemma (a proof omitted), which
 permits us to exclude from further consideration non-ergodic measures
  in proving the lower bound \eqref{low}.

\begin{lemma}\label{specification}
Let $X$ be transitive 
and $\phi\colon X\to\mathbb R$ a measurable function with $\sup\phi<\infty$
and $\sup_{n\geq1} D_n(\phi)<\infty$.
For any $\sigma$-invariant measure 
$\mu\in\mathcal M$ with finite entropy
 there exists a sequence $\{\mu_k\}$ of ergodic measures in $\mathcal M_\phi(\sigma)$
such that $\mu_k\to\mu$ in the weak*-topology, $h(\mu_k)\to h(\mu)$
and $\int\phi d\mu_k\to\int \phi d\mu$.
\end{lemma}

 \subsection{Key lower bound}\label{keyest}
 For an integer $l\geq1$, $\varphi_j\in C_u(X)$ and $\alpha_j\in\mathbb R$ for $j=1,\ldots,l$ 
consider an weak*-open set
\begin{equation*}\mathcal V\{\varphi_j,\alpha_j\}_{j=1,\ldots,l}= \left\{ \mu\in \mathcal M\colon
\int\varphi_jd\mu>\alpha_j\quad\text{for }j=1,\ldots,l\right\}.\end{equation*}
\begin{prop}\label{lowbound}
Let $X$ be transitive,
 $\phi\colon X\to\mathbb R$ a measurable function 
and $\mu_\phi$  a Gibbs state for the potential $\phi$.
Let $l\geq1$ be an integer, $\varphi_j\in C_u(X)$ and $\alpha_j\in\mathbb R$ for $j=1,\ldots,l$.
Then 
\begin{equation*}
\liminf_{n\to\infty} \frac{1}{n} \log \xi_n(\mathcal V\{\varphi_j,\alpha_j\}_{j=1,\ldots,l})
\geq \sup\{F(\mu)\colon\mu\in \mathcal V\{\varphi_j,\alpha_j\}_{j=1,\ldots,l}\}.
\end{equation*}
If moreover $X$ is finitely primitive, then the same inequality continues to hold
with $\xi_n$ replaced by $\eta_n,$ $\eta_{a,n}$, $\zeta_{y,n}$ and $\zeta_{a,y,n}$ with
 $a\in \mathbb N$, $y\in X$.
\end{prop}


\begin{proof}
Write $\mathcal V$ for $\mathcal V\{\varphi_j,\alpha_j\}_{j=1,\ldots,l}$
and
let $\mu\in\mathcal V$.
If $\mu\notin\mathcal M_\phi(\sigma)$ then $F(\mu)=-\infty$. 
  Assume $\mu\in\mathcal M_{\phi}(\sigma)$.
 By \cite[Theorem 1.4]{MauUrb01},
  $P(\phi)<\infty$ implies $h(\mu)<\infty$.
By virtue of Lemma \ref{specification} we may assume $\mu$ is ergodic.
Let $\epsilon>0$ satisfy
$\int\varphi_j d\nu-\epsilon>\alpha_j$ for $j=1,\ldots,l$.
Let $n_0>1$ and for each integer $n\geq n_0$
let $F^n$ be the finite subset of $E^n$ for which the conclusion of Lemma \ref{katok''} holds
for $\varphi_j$ $(j=1,\ldots,l)$ and $\mu$: 
$\#F^n\geq\exp((h(\mu)-\epsilon)n)$;
$\inf_{[F^n]} S_n\phi
\geq\left(\int\phi d\mu-\epsilon\right)n$;
$[F^n]\subset
\left\{ x\in X\colon \delta_x^n\in\mathcal V\right\}$.
The last inclusion is a consequence of Lemma \ref{katok''}(b) and the choice of $\epsilon$.

It is convenient to split the rest of the proof of Proposition \ref{lowbound} into three steps,
corresponding to the sequences of distributions.
  \medskip

\noindent{\it Step 1. (Lower bound for empirical means).} 
For every $n\geq n_0$ and every $w\in F^n$ we have 
\begin{equation*}
 \begin{split}
\mu_\phi[w]&\geq c_0^{-1}e^{-P(\phi)n} \inf_{[w]} \exp S_n\phi\\
&\geq c_0^{-1}e^{-P(\phi)n}\exp\left(
 \left(\int\phi d\mu-\epsilon\right)n\right).\end{split}\end{equation*}
Summing this inequality over all $w\in F^n$ yields
\begin{align*} \frac{1}{n} \log\mu_\phi\left\{x\in X\colon\delta_x^n\in\mathcal V\right\}&
\geq\frac{1}{n}\log\left(\#F^n\inf_{w\in F^n}\mu_\phi[w]\right)\\
&\geq -P(\phi)+h(\mu)+\int\phi d\mu-2\epsilon-\frac{1}{n}\log c_0.
 \end{align*}
 Letting $n\to\infty$ and then $\epsilon\to 0$ yields the desired inequality.
\medskip


\noindent{\it Step 2. (Lower bound for weighted periodic points).}
Assume $X$ is finitely primitive. 
Let $\Lambda\subset E^*$ be the finite set and $N\geq0$ the integer given by the finite primitiveness.
Let $a\in \mathbb N$ and $n> n_0+2N$ an integer.
For each $w\in F^{n-2N-1}$ fix $\kappa=\kappa(w)\in\Lambda$, $\rho=\rho(w)\in \Lambda$ with $ a\kappa w\rho a\in E^{n+1}$.
The $n$-cylinder $[a\kappa w\rho]$ contains exactly one point from $[a]\cap{\rm Per}_n(\sigma)$.
Since each function $\varphi_j$ is bounded and $\Lambda$, $N$ are independent of $n$,
$\inf_{[a\kappa w\rho]} S_n\varphi_j>\alpha_jn$ holds for sufficiently large $n$.
 Hence
 $\delta_x^n\in\mathcal V$ holds
  for every $x\in  [a\kappa w \rho]$. Therefore
  \begin{align*}
\sum_{\stackrel{x\in [a]\cap{\rm Per}_n(\sigma)}{\delta_x^n\in\mathcal V}} \exp S_n\phi(x)
&\geq\sum_{w\in F^{n-2N-1}} \inf_{[a\kappa w\rho]}\exp S_n\phi\\
&\geq \left(\inf_{[\Lambda]}\exp S_N\phi\right)^2
\inf_{[a]}\exp\phi\sum_{w\in F^{n-2N-1}} \inf_{[w]}\exp S_{n-2N-1}\phi\\
&\geq \left(\inf_{[\Lambda]}\exp S_N\phi\right)^2
\inf_{[a]}\exp\phi\#F^{n-2N-1}\inf_{[F^{n-2N-1}]}\exp S_{n-2N-1}\phi.
\end{align*}
For sufficiently large $n$, we apply the estimates on $\mu$ to the last factor to get
\begin{align*}
\frac{1}{n} \log \eta_{a,n} (\mathcal V)=&\frac{1}{n}
\log\left(\frac{1}{Z_{n}(\phi,[a]\cap{\rm Per}_n(\sigma))}
\sum_{\stackrel{x\in[a]\cap{\rm Per}_n(\sigma)}{\delta_x^n\in   \mathcal V   }}
\exp S_n\phi(x)\right)\\
  \geq& -\frac{1}{n}\log Z_{n}(\phi,[a]\cap{\rm Per}_n(\sigma))+h(\mu)+\int\phi d\mu-2\epsilon.\end{align*}
 As $n\to\infty$, the first term of the last line converges to $-P(\phi)$ by Proposition \ref{pressure}. 
  Then letting $\epsilon\to0$ yields the desired inequality for $\eta_{a,n}$.
  Since ${\rm Per}_n(\sigma)$ contains $[a]\cap{\rm Per}_n(\sigma)$ and
 $\lim(1/n)\log Z_n(\phi,{\rm Per}_n(\sigma))= P(\phi)$ by Proposition \ref{pressure}, 
 the lower bound for $\eta_n$ also follows.
  \medskip
  
\noindent{\it Step 3. (Lower bound for weighted iterated pre-images).} 
Assume $X$ is finitely primitive and let $\Lambda$, $N$ be the same as in Step 2.
Let $a\in\mathbb N$, $y=(y_i)_{i\in\mathbb N}\in X$ and $n>n_0+2N$ an integer.
For each $w\in F^{n-2N-1}$ fix $\kappa=\kappa(w)\in\Lambda$, $\rho=\rho(w)\in \Lambda$
  with $a\kappa w\rho y_0\in E^{n+1}$. The $n$-cylinder $[a\kappa w\rho]$
  contains exactly one point from $\sigma^{-n}y$.
In the same way as in Step 2 we have
  \begin{align*}
\sum_{\stackrel{x\in \sigma^{-n}y}{\delta_x^n\in\mathcal V}} \exp S_{n}\phi(x)
\geq \left(\inf_{[\Lambda]}\exp S_N\phi\right)^2
\inf_{[a]}\exp\phi\#F^{n-2N-1}\inf_{[F^{n-2N-1}]}\exp S_{n-2N-1}\phi.
\end{align*}
For sufficiently large $n$, 
\begin{align*}
\frac{1}{n} \log \zeta_{a,y,n} (\mathcal V)\geq-\frac{1}{n}\log Z_{n}(\phi,[a]\cap\sigma^{-n}y)+h(\mu)+\int\phi d\mu-2\epsilon.\end{align*} 
 As $n\to\infty$, the first term of the right-hand side converges to $-P(\phi)$ by Proposition \ref{pressure}. 
  Then letting $\epsilon\to0$ yields the desired inequality for $\zeta_{a,y,n}$.
  Since $\sigma^{-n}y$ contains $[a]\cap\sigma^{-n}y$ and
 $\lim(1/n)\log Z_n(\phi,\sigma^{-n}y)= P(\phi)$ by Proposition \ref{pressure}, 
 the lower bound for $\zeta_{y,n}$ also follows.
\end{proof}

\subsection{End of proof of the lower bound}\label{endrow}
It is now straightforward to finish the proof of the lower bound \eqref{low} for all open sets.

 \begin{proof}[Proof of the lower bound  for open sets.]
 Let $X$ be finitely irreducible and $\mu_\phi$  a Gibbs state for a measurable potential $\phi$.
 Let $\mathcal V$ be an open subset of $\mathcal M$ of the form in 
 Proposition \ref{lowbound}. Then $$ \liminf_{n\to\infty}\frac{1}{n}\log
\xi_n({\mathcal V}) \ge 
\sup_{\mathcal V}F.$$
These open sets form a base of the weak*-topology on $\mathcal M$.
Let $\mathcal G$ be an arbitrary open subset of $\mathcal M$.
Take a subset $\{\mathcal V_\gamma\}_{\gamma}$ of this base with
$\mathcal G=\bigcup_{\gamma}\mathcal V_\gamma$.
We have \begin{equation*}
\liminf_{n\to\infty}\frac{1}{n}\log\xi_n(\mathcal G)
\ge \sup_{\gamma} 
\sup_{\mathcal V_\gamma}F=\sup_{\mathcal G} F=
-\inf_{\mathcal G}I.
\end{equation*}
If $X$ is finitely primitive, then the same reasoning 
yields \eqref{low} for all open sets and for all sequences of distributions
other than $\{\xi_n\}$. \end{proof}

 \section{Large deviations upper bound}\label{upbound}
 All that remains to show is the upper bound  \eqref{up} for all closed sets.
 In $\S$\ref{exptight} we show the exponential tightness of the sequences
of probability measures appearing in Theorems A, B and C.
 Based on a preliminary result in $\S$\ref{next}
  we prove a key upper bound in $\S$\ref{key}. Combining this bound
   with the exponential tightness 
we obtain the upper bound for all closed sets, completing the proofs
of all the theorems in $\S$\ref{completes}.

\subsection{Exponential tightness}\label{exptight}

To obtain the upper bound  for non-compact closed sets requires a way 
of showing that most of the probability masses (at least on an exponential scale) 
is concentrated on compact sets. 
A precise statement is as follows.

\begin{prop}\label{tight}
Let $\phi\colon X\to \mathbb R$ be a measurable function and $\mu_\phi$ a 
Gibbs state for the potential $\phi$. Then $\{\xi_n\}$ is exponentially tight, i.e.,
for every $L>0$
there exists a compact set $\mathcal K_L\subset\mathcal M$
such that $$\limsup_{n\to\infty}\frac{1}{n}\log\xi_n(\mathcal K_L^c)\leq-L.$$
If moreover $X$ is finitely primitive, then $\eta_n$, $\eta_{a,n}$, $\zeta_{y,n}$ and $\zeta_{a,y,n}$
with $a\in \mathbb N$, $y\in X$ are exponentially tight.
\end{prop}

\begin{proof}
The proof of Proposition \ref{tight} consists of four steps.
In Step 1 we prove a key recurrence estimate relative to certain compact subsets of $X$.
In subsequent steps, we use this estimate
to construct compact subsets of $\mathcal M$ as in the statement of Proposition \ref{tight} 
for each sequence of distributions.
\medskip

 \noindent{\it Step 1. (Recurrence estimate relative to compact subsets of $X$).}\label{interm}
Let $\theta\in\mathbb R$ be such that
\begin{equation}\label{thet}
0<\theta\leq\min\left\{c_0^{-3},\frac{1}{5}\right\},
\end{equation}
where $c_0$ is in \eqref{Gibbs}.
Let $\{N_i\}_{i\in\mathbb N}$ be
a non-decreasing sequence in $\mathbb N$ such that
\begin{equation}\label{assumption}
\sum_{k=N_i+1}^{\infty}\mu_\phi[k]\leq\theta^{i+1}\ \text{ for every $i\in\mathbb N$}.\end{equation}
Define 
$$\Gamma=\{x\in X\colon x_i\leq N_i\quad\text{for every }i\in\mathbb N\},$$
which is a compact subset of $X$.



\begin{lemma}\label{expo}
For every integer $n\geq1$ and $m=1,\ldots, n$,
\begin{equation*}
\mu_\phi\left\{x\in X\colon \delta_x^n(\Gamma^c)=\frac{m}{n}\right\}\leq
\frac{2^n(4\theta)^m}{1-4\theta}.\end{equation*}
\end{lemma}
In other words, the $\mu_\phi$-measure of the set of points which 
visit the complement of $\Gamma$ exactly $m$-times up to time $n-1$ decays exponentially in $m$.
The proportional length of the time interval in which this exponential decay is not apparent 
due to the factor $2^n$ can be made arbitrarily short  by choosing sufficiently small $\theta$ and then choosing an appropriate
 $\{N_i\}_{i\in\mathbb N}$.

\begin{proof}[Proof of Lemma \ref{expo}.]
For each $x\in\Gamma^c$ define
$p(x)=\min\{i\in\mathbb N\colon x_i>N_i\}.$
Since $\{N_i\}$ is non-decreasing,
 $x\in \Gamma^c$ implies 
 $\sigma^{i}x\in\Gamma^c$
for $i=0,\ldots, p(x)$. 
 Define a sequence $0\leq n_1<n_2<\cdots$ of integers inductively as follows:
$n_1=\min\{i\geq0\colon \sigma^ix\in\Gamma^c\}$;
 $n_{j+1}=\min\{i>n_j+p(\sigma^{n_j}x)\colon \sigma^ix\in \Gamma^c\}$
 for $j\geq1$.
 The $n_j$ are called {\it free return times of $x$}.
 Put $p_j=p(\sigma^{n_j}(x))+1$ and call it {\it the depth of $n_j$}.
 Notice that $0\leq n_1<n_1+p_1\leq n_2<n_2+p_2\leq n_3<\cdots.$

Let $s\geq1$ be an integer and  $j\in\{1,\ldots,s\}$.
For two $j$-strings $n_1\cdots n_j$, $p_1\cdots p_j$ of integers
with $0\leq n_1<\cdots <n_j$
and $p_1,\ldots,p_j\geq1,$
define $\Gamma_{n_1\cdots n_j}^{p_1\cdots p_j}$ to be the set of $x\in X$ for which
 $n_1,\ldots,n_j$ are all the free return times in $[0, n_j]$, with $p_1,\ldots,p_j$ the 
 corresponding depths.
The sequence $\{\Gamma_{n_1\cdots n_j}^{p_1\cdots p_j}\}_{j=1}^s$ of sets is decreasing in $j$.

By induction we show 
\begin{equation}\label{measure}
\mu_\phi(\Gamma_{n_1\cdots n_j}^{p_1\cdots p_j})\leq \theta^{p_1+\cdots +p_j}
\ \ \text{for $j=1,\ldots, s$.}\end{equation}
Start with $j=1$.
If $n_1=0$ then $p_1=1$ and
 by \eqref{assumption} with $i=0$,
 $$\mu_\phi(\Gamma_{n_1}^{p_1})\leq\sum_{k=N_0+1}^\infty\mu_\phi[k]\leq\theta=\theta^{p_1}.$$
If $n_1>0$, 
 let $w\in E^{n_1+p_1-1}$ be such that the corresponding cylinder 
 $[w]$ intersects $\Gamma_{n_1}^{p_1}$.
 Any point in 
$\Gamma_{n_1}^{p_1}$ is contained in such a cylinder.
By Lemma \ref{distor}(b) and \eqref{assumption} with $i=p_1$,
$$\frac{\mu_\phi([w]\cap\Gamma_{n_1}^{p_1})}{\mu_\phi[w]}\leq 
\frac{\sum_{k=N_{p_1}+1}^\infty\mu_\phi[wk]}{\mu_\phi[w]}\leq c_0^3
\sum_{k=N_{p_1}+1}^\infty\mu_\phi[k]\leq c_0^3\theta^{p_1+1}\leq\theta^{p_1}.$$
Rearranging this inequality and summing the result over all 
$w$ yield
\begin{equation*}\label{first}\mu_\phi(\Gamma_{n_1}^{p_1})=
\sum_{w}\mu_\phi([w]\cap\Gamma_{n_1}^{p_1})\leq
\theta^{p_1}\sum_{w}\mu_\phi[w]\leq
 \theta^{p_1}.\end{equation*}
 Hence, \eqref{measure} holds for $j=1$.
 
 Proceeding to the general step of induction,
 let $s>1$ and $j\in\{1,\ldots,s-1\}$.
 Let $w\in E^{n_{j+1}+p_{j+1}-1}$ be such that the corresponding cylinder
$[w]$ is contained in $\Gamma_{n_1\cdots n_j}^{p_1\cdots p_j}$ and 
intersects $\Gamma_{n_1\cdots n_{j+1}}^{p_1\cdots p_{j+1}}$.
Any point in 
$\Gamma_{n_1\cdots n_{j+1}}^{p_1\cdots p_{j+1}}$ is contained in such a cylinder.
By Lemma \ref{distor}(b) and \eqref{assumption} with $i=p_{j+1}$,
\begin{equation*}\label{eqlast}
\begin{split}
\frac{\mu_\phi([w]\cap  \Gamma_{n_1\cdots n_{j+1}}^{p_1\cdots p_{j+1}})}{\mu_\phi[w]}
\leq\frac{\sum_{k=N_{p_{j+1}}+1}^\infty\mu_\phi[wk]}{\mu_\phi[w]}
\leq c_0^3\sum_{k=N_{p_{j+1}}+1}^\infty\mu_\phi[k]\leq c_0^3\theta^{p_{j+1}+1}
\leq\theta^{p_{j+1}}.
\end{split}
\end{equation*}
Rearranging this and summing the result over all $w$ yield
$$\mu_\phi(\Gamma_{n_1\cdots n_{j+1}}^{p_1\cdots p_{j+1}})\leq\theta^{p_{j+1}}
\mu_\phi( \Gamma_{n_1\cdots n_{j}}^{p_1\cdots p_{j}}),$$
which recovers the assumption of the induction.

Let $n\geq1$ be an integer and $m\in\{1,\ldots, n\}$.
Notice that
$$\mu_\phi\left\{x\in X\colon\delta_x^n(\Gamma^c)=\frac{m}{n}\right\}\leq\sum_{s=1}^m
\sum_{P=m}^\infty\sum_{\stackrel{(p_1,\ldots,p_s)}
{\sum_{j=1}^s p_j=P}}\sum_{
\stackrel{(n_1,\ldots,n_s)}{0\leq n_1<\cdots<n_s\leq n-1}}  \mu_\phi(\Gamma_{n_1\cdots n_s}^{p_1\cdots p_s}).$$
For each fixed $s\in\{1,\ldots,m\}$, the number of ways of locating free return times $n_1,\ldots,n_s$
in $[0,n]$ is 
$\left(\begin{smallmatrix}n\\s\end{smallmatrix}\right)\leq 2^n$.
For each location $(n_1,\ldots,n_s)$ of free return times,
the number of all feasible combinations of depths $(p_1,\ldots,p_s)$ with $\sum_{j=1}^s p_j=P$ is bounded by the number of ways
of dividing $P$-objects into $s$-groups, and so
$\left(\begin{smallmatrix}P+s-1\\s-1\end{smallmatrix}\right)\leq 2^{P+s-1}$.
This and \eqref{measure} yield
\begin{align*}\mu_\phi\left\{x\in X\colon\delta_x^n(\Gamma^c)=\frac{m}{n}\right\}&\leq 2^n\sum_{s=1}^m
\sum_{P=m}^\infty 2^{P+s-1}\theta^P\\
&\leq2^n\sum_{P=m}^\infty (4\theta)^P\\
&=\frac{2^n(4\theta)^m}{1-4\theta},\end{align*}
as required.
\end{proof}

\noindent{\it Step 2. (Exponential tightness for empirical means).}
We adapt a portion of the proof of Sanov's Theorem \cite[Lemma 6.2.6]{DemZei98} 
to show the exponential tightness for $\{\xi_n\}$, using
 Lemma \ref{expo} to compensate the lack of independence in our setting.

For each integer $\ell\geq1$ fix $\theta\in\mathbb R$ 
such that \eqref{thet} holds and
\begin{equation}\label{theta}\frac{1}{1-4\theta}\sum_{m=0}^\infty e^{2\ell^2m}(4\theta)^m\leq2.\end{equation}
Fix a non-decreasing integer sequence $\{N_i\}_{i\in\mathbb N}$ 
satisfying \eqref{assumption}.
Define a compact set
$$\Gamma_\ell=\{x\in X\colon x_i\leq N_i\text{ for every }i\in\mathbb N\},$$
and $$\mathcal K^\ell=\left\{\nu\in\mathcal M\colon \nu(\Gamma_\ell)\geq1-\frac{1}{\ell}\right\}.$$
Since $\mathcal M$ is a Polish space and $\Gamma_\ell$ is a closed set, by Portmanteau's Theorem 
the weak*-convergence
$\mu_k\to\mu$ for a sequence $\{\mu_k\}$ in $\mathcal K^\ell$ of probability measures
implies
$\displaystyle{\limsup\mu_k(\Gamma_\ell)}\leq\mu(\Gamma_\ell)$.
Hence, $\mathcal K^\ell$ is a closed set.
For each integer $L\geq1$ define
$$\mathcal K_L=\bigcap_{\ell=L}^\infty \mathcal K^\ell.$$
By the compactness of each $\Gamma_\ell$, $\mathcal K_L$ is tight and 
any sequence contained in it has a limit point by Prohorov's Theorem.
As $\mathcal K_L$ is closed, it is sequentially compact.
Since the weak*-topology on $\mathcal M$ is metrizable, $\mathcal K_L$ is a compact
subset of $\mathcal M$.
For every $n\geq1$,
\begin{align*}
\mu_\phi\{x\in X\colon\delta_x^n\notin \mathcal K^\ell\}&=\mu_\phi\left\{x\in X\colon\delta_x^n(\Gamma_\ell^c)\geq\frac{1}{\ell}\right\}\\
&=\mu_\phi\left\{x\in X\colon \exp\left(\ell^2n \delta_x^n(\Gamma_\ell^c)\right)\geq e^{\ell n}\right\}\\
&\leq e^{-2\ell n}\int_X \exp\left(2\ell^2n\delta_x^n(\Gamma_\ell^c)\right)d\mu_\phi(x)\\
&=e^{-2\ell n}\sum_{m=0}^{n}e^{2\ell^2m}\mu_\phi\left\{x\in X\colon \delta_x^n(\Gamma_\ell^c)=\frac{m}{n}\right\}\\
&\leq \frac{2^ne^{-2\ell n}}{1-4\theta}\sum_{m=0}^{n}e^{2\ell^2m}
(4\theta)^m\\
&\leq2^{n+1}e^{-2\ell n}.
\end{align*}
We have used Chebyshev's bound for the first inequality,
Lemma \ref{expo} for the second one and \eqref{theta} for the last one.
For $L$ large enough,
$$\xi_n(\mathcal K_L^c)\leq\sum_{\ell=L}^\infty\mu_\phi\{x\in X\colon\delta_x^n\notin \mathcal K^\ell\}
\leq2^{n+1}\sum_{\ell=L}^{\infty} e^{-2\ell n}\leq e^{-Ln}.$$
This yields
$\limsup(1/n)\log\xi_n(\mathcal K_L^c)\leq-L$ as required.
\medskip

\noindent{\it Step 3. (Exponential tightness for weighted periodic points).}
We assume $X$ is finitely primitive, and show the 
exponential tightness for $\{\eta_n\}$ and $\{\eta_{a,n}\}$.
For each integer $\ell\geq1$ fix $\theta\in\mathbb R$ such that \eqref{thet} holds and
\begin{equation}\label{theta'}2(4\theta)^{\frac{1}{2\ell}}\leq e^{-2\ell}.\end{equation}
Fix a non-decreasing integer sequence $\{N_i\}_{i\in\mathbb N}$ 
satisfying \eqref{assumption}.
As in Step 2, define
$\Gamma_\ell=\{x\in X\colon x_i\leq N_i\text{ for every }i\in\mathbb N\}$
and $\mathcal K^\ell=\left\{\nu\in\mathcal M\colon \nu(\Gamma_\ell)\geq1-1/\ell\right\}.$
For each integer $L\geq1$ the set
$\mathcal K_L=\bigcap_{\ell=L}^\infty \mathcal K^\ell$
  is compact for the same reason as in Step 2.
For every $n\geq1$,
\begin{align*}
\eta_n(\mathcal M\setminus\mathcal K^\ell)&=\frac{1}{Z_{n}(\phi,{\rm Per}_n(\sigma))}\sum_{\stackrel{x\in{\rm Per}_n(\sigma)}{\delta_x^n\notin\mathcal K^\ell}}
\exp S_n\phi(x)\\
&=\frac{1}{Z_{n}(\phi,{\rm Per}_n(\sigma))}\sum_{\stackrel{x\in{\rm Per}_n(\sigma)}{\delta_x^n(\Gamma_\ell^c)
\geq1/\ell}}\exp S_n\phi(x)\\
&\leq \frac{1}{Z_{n}(\phi,{\rm Per}_n(\sigma))}\sum_{m=\lfloor n/\ell\rfloor}^{n}\sum_{\stackrel{x\in{\rm Per}_n(\sigma)}{\delta_x^{n}(\Gamma_\ell^c)=m/n}} \exp 
S_n\phi(x)\\
&\leq
\frac{c_0e^{P(\phi)n}}{Z_{n}(\phi,{\rm Per}_n(\sigma))}\sum_{m=\lfloor n/\ell\rfloor}^{n}
\sum_{\stackrel{x\in{\rm Per}_n(\sigma)}{\delta_x^{n}(\Gamma_\ell^c)=m/n}}\mu_\phi[x_0,\ldots,x_{n-1}]\quad\text{by }\eqref{Gibbs}.
\end{align*}
Lemma \ref{expo} gives
\begin{align*}
\sum_{\stackrel{x\in{\rm Per}_n(\sigma)}{\delta_x^{n}(\Gamma_\ell^c)=m/n}}\mu_\phi[x_0,\ldots,x_{n-1}]
&\leq\mu_\phi\left\{x\in X\colon \delta_x^n(\Gamma_\ell^c)=\frac{m}{n}\right\}\\
&\leq \frac{2^n(4\theta)^m}{1-4\theta}.
\end{align*}
Plugging this into the above inequality and then using \eqref{theta'} give
\begin{align*}
\eta_n(\mathcal M\setminus\mathcal K^\ell)&\leq
\frac{2^nc_0}{1-4\theta}
\frac{e^{P(\phi)n}}{Z_{n}(\phi,{\rm Per}_n(\sigma))}
\sum_{m=\lfloor n/\ell\rfloor}^{n}(4\theta)^m\\
&\leq\frac{2^nc_0}{(1-4\theta)^2}\frac{e^{P(\phi)n}}{Z_{n}(\phi,{\rm Per}_n(\sigma))}(4\theta)^{\lfloor n/\ell\rfloor}\\
&\leq
\frac{c_0}{(1-4\theta)^2}\frac{e^{P(\phi)n}}{Z_{n}(\phi,{\rm Per}_n(\sigma))}e^{-2\ell n}.
\end{align*}
The last inequality holds provided $n>2L$.
For $L$ large enough,
\begin{align*}\eta_n( \mathcal K_L^c)&\leq\sum_{\ell=L}^\infty\eta_n(\mathcal M
\setminus\mathcal K^\ell)\\
&\leq\frac{c_0}{(1-4\theta)^2}\frac{e^{P(\phi)n}}{Z_{n}(\phi,{\rm Per}_n(\sigma))}\sum_{\ell=L}^{\infty} e^{-2\ell n}\\
&\leq\frac{c_0}{(1-4\theta)^2}\frac{e^{P(\phi)n}}{Z_{n}(\phi,{\rm Per}_n(\sigma))} e^{-Ln}.\end{align*}
Proposition \ref{pressure} gives $\lim(1/n)\log Z_{n}(\phi,{\rm Per}_n(\sigma))= P(\phi)$,
 and thus we obtain
$\limsup(1/n)\log\eta_n(\mathcal K_L^c)\leq -L$ as required.
The exponential tightness for $\{\eta_{a,n}\}$ $(a\in\mathbb N)$ follows from simply replacing
${\rm Per}_n(\sigma)$ in the above formulas by $[a]\cap{\rm Per}_n(\sigma)$.
\medskip

\noindent{\it Step 4. (Exponential tightness for weighted iterated pre-images).}
We assume $X$ is finitely primitive and show the 
exponential tightness for $\{\zeta_{y,n}\}$ and $\{\zeta_{a,y,n}\}$.
In the same way as in Step 3, we have
\begin{align*}
\zeta_{y,n}(\mathcal M\setminus\mathcal K^\ell)&\leq
\frac{c_0e^{P(\phi)n}}{Z_{n}(\phi,\sigma^{-n}y)} \sum_{m=\lfloor n/\ell\rfloor}^{n}
\sum_{\stackrel{x\in\sigma^{-n}y}{\delta_x^{n}(\Gamma_\ell^c)=m/n}}\mu_\phi[x_0,\ldots,x_{n-1}]\\
&\leq
\frac{c_0}{(1-4\theta)^2}\frac{e^{P(\phi)n}}{Z_{n}(\phi,\sigma^{-n}y)}e^{-2\ell n}.
\end{align*}
Hence,
for $L$ large enough and $n>2L$ the same upper bound as in Step 3 is available on $\zeta_{y,n}( \mathcal K_L^c).$
Proposition \ref{pressure} gives $\lim(1/n)\log Z_{n}(\phi,\sigma^{-n}y)=P(\phi)$, and 
we obtain $\limsup(1/n)\log\zeta_{y,n}(\mathcal K_L^c)\leq -L$ as required.
The exponential tightness for $\{\zeta_{a,y,n}\}$ ($a\in\mathbb N$, $y\in X$) follows from
simply replacing $\sigma^{-n}y$ in the above formulas by $[a]\cap\sigma^{-n}y.$
\end{proof}


\subsection{Finite Markov system}\label{next}
By {\it a finite Markov system} we mean a pair $(a,G^n)$
with $a\in \mathbb N$, $n>1$ and $G^{n}$ a finite subset of $E^{n}(a,a)$.
The next lemma is proved along the line of the thermodynamic formalism
 for finite Markov shifts \cite{Bow75,Rue78}.

 \begin{lemma}\label{horse}
 Let $\phi\colon X\to\mathbb R$ be a measurable function and 
  $\mu_\phi$ a Gibbs state for the potential $\phi$.
  Let $l\geq1$ be an integer, 
 $\varphi_j\colon X\to\mathbb R$ continuous 
and $\alpha_j\in \mathbb R$ for $j=1,\ldots,l$. 
There exists $n_{1}>1$ such that the following holds: let
$n\geq n_{1}$ be an integer and 
$(a,G^{n})$ a finite Markov system satisfying
 $\inf_{[G^{n}]}  S_{n-1}\varphi_j>\alpha_j(n-1)$ for $j=1,\ldots,l$.
There exists a $\sigma$-invariant measure $\mu\in\mathcal M$ which is supported on a compact set 
and satisfies
$$ 
\log  \mu_\phi[G^{n}]  \leq F(\mu)(n-1)
+\log c_0$$
and
$$ \int\varphi_jd\mu>\alpha_j\quad\text{for }j=1,\ldots,l,$$
where $c_0$ is the constant in \eqref{Gibbs}.
\end{lemma}

 \begin{proof}
   Put $\widehat\sigma=\sigma^{n-1}$,
    $\widehat\phi=S_{n-1}\phi$ and define
 $K=\bigcap_{m\in\mathbb N}{\widehat\sigma}^{-m}[G^{n}].$
 Then $K$ is a compact set and $\widehat\sigma|_K\colon K\to K$ 
 is topologically conjugate to the full shift on $\#G^n$-symbols.
 By Lemma \ref{impose}, 
 for every integer $m\geq1$ and all 
  $x,y\in K$
 such that for each $i=0,\ldots,m-1$ there exists $\omega\in G^{n}$ with
 $\widehat\sigma^ix,\widehat\sigma^iy\in[w]$, we have
  $$ \sum_{i=0}^{m-1}\widehat
 \phi(\widehat\sigma^ix)-\widehat
 \phi(\widehat\sigma^iy)\leq D_{m(n-1)}(\phi)\leq \sup_{n\geq1}D_{n}(\phi)<\infty.$$
 Fix $z\in K$.  The variational principle \cite[Lemma 1.20]{Bow75} gives
 \begin{equation}\label{previous}\sup_{\widehat\nu\in\mathcal M(\widehat\sigma|_K)}\left(h_{\widehat\sigma|_K}(\widehat\nu)+\int\widehat\phi d\widehat\nu\right)=
 \lim_{m\to\infty}\frac{1}{m}\log\sum_{x\in (\widehat\sigma|_K)^{-m}z}\exp\left(\sum_{i=0}^{m-1}\widehat
 \phi(\widehat\sigma^ix)\right),\end{equation}
 where $\mathcal M(\widehat\sigma|_K)$ denotes 
 the space of $\widehat\sigma|_K$-invariant Borel probability measures
 endowed with the weak*-topology
 and $h_{\widehat\sigma|_K}(\widehat\nu)$ the entropy of $\widehat\nu\in \mathcal M(\widehat\sigma|_K)$
 with respect to $\widehat\sigma|_K$.
By \eqref{Gibbs},
$\inf_{[w]}\exp\widehat\phi\geq c_0^{-1}e^{P(\phi)(n-1)}\mu_\phi[w]$
holds for every $w\in G^{n}$.
Hence \begin{align*}\sum_{x\in (\widehat\sigma|_K)^{-m}z}\exp\left(\sum_{i=0}^{m-1}\widehat\phi(\widehat\sigma^ix)\right)
&\geq\left(\inf_{z'\in K}\sum_{x\in (\widehat\sigma|_K)^{-1}z'}\exp\widehat\phi(x)\right)^m\\
&\geq\left(
c_0^{-1}e^{P(\phi)(n-1)}\mu_\phi[G^{n}]\right)^m.\end{align*}
Taking logs, dividing by $m$ and letting $m\to\infty$, 
$$\lim_{m\to\infty}\frac{1}{m}\log\sum_{x\in (\widehat\sigma|_K)^{-m}z}\exp\left(
\sum_{i=0}^{m-1}\widehat\phi(\widehat\sigma^ix)\right)
\geq\log\left(c_0^{-1}e^{P(\phi)(n-1)}\mu_\phi[G^{n}]\right).$$
Plugging this into \eqref{previous} yields
$$\sup_{\widehat\nu\in\mathcal M(\widehat\sigma|_K)}\left(h_{\widehat\sigma|_K}(\widehat\nu)+\int\widehat\phi d\widehat\nu\right)\geq
\log\left(c_0^{-1}e^{P(\phi)(n-1)}\mu_\phi[G^{n}]\right).$$
Since $\mathcal M(\widehat\sigma|_K)$ is compact 
and the mapping $\widehat\nu\in\mathcal M(\widehat\sigma|_K)\mapsto h_{\widehat\sigma|_K}(\widehat\nu)+\int\widehat\phi d\widehat\nu $
is upper semi-continuous, there exists a measure $\widehat\mu\in \mathcal M(\widehat\sigma|_K)$
which attains the supremum of the left-hand side. The measure
$\mu = (1/(n-1)) \sum_{i=0}^{n-2}\widehat\mu\circ\sigma^{-i}$ is $\sigma$-invariant and
satisfies the desired properties.
 \end{proof}

\subsection{Key upper bound}\label{key}
For an integer $l\geq1$, 
 $\varphi_j\in C_u(X)$ 
and $\alpha_j\in \mathbb R$ for $j=1,\ldots,l$ denote by
$\overline{\mathcal V}\{\varphi_j,\alpha_j\}_{j=1,\ldots,l}$
the weak*-closure of $\mathcal V\{\varphi_j,\alpha_j\}_{j=1,\ldots,l}$, namely
$$\overline{\mathcal V}\{\varphi_j,\alpha_j\}_{j=1,\ldots,l}=\left\{\mu\in\mathcal M\colon \int\varphi_j
d\mu\geq\alpha_j\text{ for }
j=1,\ldots,l\right\}.$$
\begin{prop}
\label{upper0}
Let $X$ be finitely irreducible,
 $\phi\colon X\to \mathbb R$ a measurable function 
and $\mu_\phi$ a Gibbs state for the potential $\phi$.
Let $l\geq1$ be an integer, 
 $\varphi_j\in C_u(X)$ 
and $\alpha_j\in \mathbb R$ for $j=1,\ldots,l$.
For every $\epsilon>0$,
$$ \limsup_{n\to\infty}\frac{1}{n} \log \xi_{n} (\overline{\mathcal V}\{\varphi_j,\alpha_j\}_{j=1,\ldots,l})
\leq \sup\left\{F(\mu)\colon\mu\in\mathcal V\{\varphi_j,\alpha_j-\epsilon\}_{j=1,\ldots,l}\right\},$$
If moreover $X$ is finitely primitive, then
the same conclusion continues to hold with $\xi_n$ replaced by $\eta_n$, $\eta_{a,n}$,
 $\zeta_{y,n}$ and $\zeta_{a,y,n}$ with $a\in \mathbb N$, $y\in X$.
\end{prop}


\begin{proof}
It is convenient to split the proof of Proposition \ref{upper0} into three steps.
In Lemma \ref{horse} we have already shown that 
finite Markov systems can be used for bounding measures from above.
Write $\overline{\mathcal V}$ for $\overline{\mathcal V}\{\varphi_j,\alpha_j\}_{j=1,\ldots,l})$.
In order to capture the set $\overline{\mathcal V}$, in Step 1 we 
take advantage of the finite irreducibility and construct 
 finitely many finite Markov systems.
 In the remaining steps we treat each sequence of distributions separately.
 \medskip

\noindent{\it Step 1. (Reduction to finitely many finite Markov systems).}
Let $\Lambda$ be the finite subset of $E^*$ given by the finite irreducibility of $X$,
and put
$|\Lambda|=\max_{\lambda\in\Lambda}|\lambda|$.
If $\Lambda\neq\emptyset$, define $\mathbb N_\Lambda$ to be the set of $a\in \mathbb N$ for which 
there exists $\lambda=\lambda_0\cdots\lambda_{|\lambda|-1}\in\Lambda$
with $\lambda_{|\lambda|-1}=a$.
If $\Lambda=\emptyset$, put $\mathbb N_\Lambda=\{0\}$.
Set
$$c_1=\inf\{\mu_\phi[w]\colon w\in \mathbb N_\Lambda\cup \Lambda \}.$$
Since $\mu_\phi$ is a Gibbs state and $\mathbb N_\Lambda$, $\Lambda$ are finite sets, $c_1>0$ holds.
Let $\epsilon>0$ be as in Proposition \ref{upper0}.
For an integer $n>1$ and $a\in \mathbb N_\Lambda$ define 
\begin{equation}\label{hna}
H^{n}(a)=\left\{w\in E^{n}(a)\colon
\delta_z^{n}\in \overline{\mathcal V}\left\{\varphi_j,\alpha_j-\frac{\epsilon}{3}\right\}_{j=1,\ldots,l}\text{ for some }z\in[w]\right\},\end{equation}
and $$Y_n=\bigcup_{a\in \mathbb N_\Lambda}\sigma[H^{n}(a)].$$

\begin{lemma}\label{claim}
For sufficiently large integer $n>1$,
$$\mu_\phi\{x\in X\colon \delta_x^{n-1}\in \overline{\mathcal V}\}\leq
\mu_\phi(Y_n).$$
\end{lemma}
\begin{proof}
From the finite irreducibility and the definition of $\mathbb N_\lambda$,
for each $b\in \mathbb N$ there exists $a\in \mathbb N_\Lambda$ 
 with $ab\in E^*$.
It follows that for each $x\in X$ 
there exist $a\in \mathbb N_{\Lambda}$
and $y\in[a]$ with
 $x=\sigma y$. Hence, there exists an integer $n'>1$
 which depends only on $\{\varphi_j,\alpha_j\}_{j=1,\ldots,l}$, $\Lambda$ and $\epsilon$ such that
  for every $n\geq n'$ with $\delta_x^{n-1}\in \overline{\mathcal V}$,
we have \begin{align*}
S_n\varphi_j(y)=\varphi_j(y)+S_{n-1}\varphi_j(x)\geq \inf_{[\mathbb N_{\Lambda}]}\varphi_j+(n-1)\alpha_j\geq \left(\alpha_j-\frac{\epsilon}{3}\right)n,\end{align*}
  for $j=1,\ldots,l$.
This yields $y\in [H^n(a)]$, and thus $x\in Y_n$ as required.
 \end{proof}

We bound $\mu_\phi(Y_n)$ from above by constructing finitely many finite Markov systems
each based at $[a]$, $a\in\mathbb N_\Lambda$.
 In what follows, in view of Lemma \ref{varlem} we assume $n$ is large enough so that
$D_{n}(\varphi_j)\leq\epsilon n/3$ holds for $j=1,\ldots,l$.

Let $a\in \mathbb N_{\Lambda}$.
For each $w\in H^{n}(a)$ fix $\kappa=\kappa(w)\in\Lambda$ 
with $w\kappa a\in E^*$.
By the definition \eqref{hna}, for each $w\in H^{n}(a)$ there exists $z=z(w)\in [w]$ such that $S_{n}\varphi_j(z)\geq (\alpha_j-\frac{\epsilon}{3})n$
for $j=1,\ldots,l$.
For every $x\in [w\kappa]$ we have
\begin{align*}
S_{n+|\kappa|}\varphi_j(x)
&=S_{n}\varphi_j(z)+S_{n}\varphi_j(x)-S_{n}\varphi_j(z)+S_{|\kappa|}\varphi_j(\sigma^{n}x)\\
&\geq S_{n}\varphi_j(z)-D_{n}(\varphi_j)+S_{|\kappa|}\varphi_j(\sigma^{n}x)\\
&\geq \left(\alpha_j-\frac{2\epsilon}{3}\right)n+\inf_{\lambda\in\Lambda} \inf_{[\lambda]}S_{|\lambda|}\varphi_j,
\end{align*}
for $j=1,\ldots,l$.
Since $\Lambda$ is a finite set and each $\varphi_j$ is bounded, the last term of the last line is bounded. It follows that
for sufficiently large $n$,
\begin{equation}\label{upeq4}\inf_{[w\kappa]} S_{n+|\kappa|}\varphi_j>(\alpha_j-\epsilon)(n+|\kappa|),\end{equation} 
  for $j=1,\ldots,l$ and every $w\in H^n(a)$.
  
  Summing the inequality
$\mu_\phi[w]\leq c_0^3c_1^{-1}\mu_\phi[w\kappa]$ 
 over all $w\in H^{n}(a)$ which follows from
 Lemma \ref{distor}(a) gives
\begin{equation}\label{upeq0}\log\mu_\phi[H^{n}(a)]
\leq\log \sum_{w\in 
H^{n}(a)}\mu_\phi[w\kappa]+\log(c_0^3c_1^{-1}).\end{equation}
It can happen that $\#H^{n}(a)=\infty$.
Since the cylinders corresponding to the strings in $H^{n}(a)$ are pairwise disjoint and 
$\mu_\phi$ is a finite measure,
the summand of the right-hand side is bounded and
it is possible to 
choose a finite subset $B^n(a)$ of $H^{n}(a)$ such that
\begin{equation*}\label{upeq1}
\log\sum_{w\in H^{n}(a)}\mu_\phi[w\kappa]\leq\log\sum_{w\in B^n(a)}\mu_\phi[w\kappa]+1.
\end{equation*}
For each $s\in\{0,\ldots,|\Lambda|\}$
define $B_s^n(a)=\{w\in B^n(a)\colon|\kappa(w)|=s\}$.
Pick $s_0\in\{0,\ldots,|\Lambda|\}$ with
$$\sum_{w\in B^n(a)}\mu_\phi[w\kappa ]\leq(|\Lambda|+1)  \sum_{w\in B_{s_0}^n(a)}\mu_\phi[w\kappa ].$$
Combining this inequality with the previous one gives
\begin{equation}\label{upeq10}
\log\sum_{w\in H^{n}(a)}\mu_\phi[w\kappa]\leq\log\sum_{w\in B_{s_0}^n(a)}\mu_\phi[w\kappa]+\log(|\Lambda|+1)+1.
\end{equation}
Let us simply denote by {\rm const.} any constant which depends only on $X$ and $\mu_\phi$.
Since $(a, \{w\kappa a\}_{w\in B_{s_0}^n(a)})$ is a finite Markov system,
by Lemma \ref{horse} and \eqref{upeq4}
there exists a $\sigma$-invariant measure $\mu^a\in\mathcal M$ which is supported on a compact set and satisfies
\begin{equation}\label{upeq2}
\begin{split}
\log\sum_{w\in B_{s_0}^n(a)}\mu_\phi[w\kappa ]&\leq\log\sum_{w\in B_{s_0}^n(a)}\mu_\phi[w\kappa a ]+{\rm const.}\\
&\leq  F(\mu^a)(n-1)
  +{\rm const.},
  \end{split}\end{equation}
  and 
$ \int\varphi_j d\mu^a>\alpha_j-\epsilon$ for $j=1,\ldots,l$. 
From \eqref{upeq0}, \eqref{upeq10} and \eqref{upeq2} we obtain
\begin{equation}\label{spl}
\log\mu_\phi[H^{n}(a)]\leq  F(\mu^a)(n-1)+{\rm const.}
\end{equation}
\medskip

\noindent{\it Step 2. (Upper bound for empirical means).}
In Step 1 we have constructed for each $a\in\mathbb N_{\Lambda}$ a measure $\mu^a$.
Pick $\mu\in\{\mu^a\}_{a\in\mathbb N_{\Lambda}}$ with
$F(\mu)=\max_{a\in\mathbb N_\Lambda} F(\mu^a)$.
Let $n>1$ be a large integer for which $\xi_{n-1}(\overline{\mathcal V})>0$
holds. Then
\begin{align*}\log\xi_{n-1}(\overline{\mathcal V})&=\log\mu_\phi
\{x\in X\colon \delta_x^{n-1}\in \overline{\mathcal V}\}\\
&\leq\log\mu_\phi(Y_n)\quad\text{by Lemma \ref{claim}}\\
&\leq\log\sum_{a\in\mathbb N_\Lambda}\mu_\phi[H^n(a)]+{\rm const.}\quad\text{by Lemma \ref{distor}(a)}\\
&\leq F(\mu)(n-1)+\log\#\mathbb N_{\Lambda}+{\rm const.}\quad\text{by \eqref{spl}},
\end{align*}
which implies the desired inequality in Proposition \ref{upper0} for $\xi_n$.
\medskip

\noindent{\it Step 3. (Upper bounds for weighted periodic points and iterated pre-images).}
Assume $X$ is finitely primitive.
Let $n\geq1$ be a large integer for which
$\eta_n (\overline{\mathcal V})>0$ holds.
Then
$\xi_n (\overline{\mathcal V}\{\varphi_j,\alpha_j-\epsilon/2\}_{j=1,\ldots,l})>0$,
and thus the argument in Step 1 with $\alpha_j$ replaced by $\alpha_j-\epsilon/2$ works. 
We have
\begin{align*}
 \eta_n (\overline{\mathcal V})=&\frac{1}{Z_{n}(\phi,{\rm Per}_n(\sigma))}
\sum_{\stackrel{x\in{\rm Per}_n(\sigma)}{\delta_x^n\in  
 \overline{\mathcal V}}}
 \exp S_n\phi(x)\\
 \leq&\frac{1}{Z_{n}(\phi,{\rm Per}_n(\sigma))}
\sum_{x\in{\rm Per}_n(\sigma)\cap Y_n}
 \exp S_n\phi(x)\quad\text{by Lemma \ref{claim}}\\
  \leq&\frac{c_0e^{P(\phi)n}}{Z_{n}(\phi,{\rm Per}_n(\sigma))}
\sum_{x\in{\rm Per}_n(\sigma)\cap Y_n}
\mu_\phi[x_0,\ldots,x_{n-1}]\quad\text{by \eqref{Gibbs}.}
\end{align*}
From the definition of $\mathbb N_{\Lambda}$, 
for each $x=(x_i)_{i\in\mathbb N}\in{\rm Per}_n(\sigma)\cap Y_n$ there exists $a\in \mathbb N_{\Lambda}$
such that $ax_0\cdots x_{n-1}\in H^{n+1}(a)$.
Lemma \ref{distor}(a) gives
$\mu_\phi[x_0,\ldots,x_{n-1}]\leq c_0^3c_1^{-1}\mu_\phi[a,x_0,\ldots,x_{n-1}]$.
Summing this over all $x\in{\rm Per}_n(\sigma)\cap Y_n$
and using \eqref{spl},
\begin{align*}
\log\sum_{x\in{\rm Per}_n(\sigma)\cap Y_n}
\mu_\phi[x_0,\ldots,x_{n-1}]\leq& \log\sum_{a\in \mathbb N_{\Lambda}}\mu_\phi[H^{n+1}(a)]+
{\rm const.}\\
\leq&F(\mu)n
+{\rm const.}\end{align*}
For sufficiently large $n$,
\begin{align*}
\frac{1}{n} \log \eta_n (\overline{\mathcal V})  \leq& -\frac{1}{n}\log Z_{n}(\phi,{\rm Per}_n(\sigma))+
P(\phi)+\frac{1}{n}\log c_0 \\&+
 \frac{1}{n}\log\sum_{x\in{\rm Per}_n(\sigma)\cap Y_n}
 \mu_\phi[x_0,\ldots,x_{n-1}]\\
 \leq&  -\frac{1}{n}\log Z_{n}(\phi,{\rm Per}_n(\sigma))+P(\phi)+F(\mu)+\frac{\rm const}{n}.\end{align*}
As $n\to\infty$, the first term converges to $P(\phi)$ by Proposition \ref{pressure}
and so the desired inequality in Proposition \ref{upper0} holds for $\eta_n$.
 That for $\eta_{a,n}$ ($a\in \mathbb N$) is obtained simply by
 replacing ${\rm Per}_n(\sigma)$ in the above formulas by $[a]\cap{\rm Per}_n(\sigma)$.
 Proofs for the distributions $\zeta_{y,n}$, $\zeta_{a,y,n}$ 
($a\in \mathbb N$, $y\in X$) are analogous and omitted.
 \end{proof}
 
 \subsection{End of proof of the upper bound}\label{completes}
 We are in position to finish the proof of the upper bound \eqref{up} for all closed sets and complete
 the proofs of all the theorems.
 
 \begin{proof}[Proof of the upper bound for closed sets]

Let $X$ be finitely irreducible,
 $\phi\colon X\to\mathbb R$ a measurable function and
 $\mu_\phi$ a Gibbs state for the potential $\phi$.
By virtue of the exponential tightness in Proposition \ref{tight},
we have only to consider compact closed sets (see \cite[Lemma 1.2.18(a)]{DemZei98}, \cite[Theorem 2.19]{RasSep}).

Let $\mathcal{K}\subset\mathcal M$ be a compact set.
Let $\mathcal G$ be an open set containing $\mathcal K$. 
Since the weak*-topology is metrizable and $\mathcal K$ is compact,
it is possible to choose a finite number of open sets
$\mathcal V_1,\ldots,\mathcal V_r$ of the form
$\mathcal V_k=\mathcal V\{\varphi_j,\alpha_j\}_{j=1,\ldots,l}$ satisfying
$\mathcal K\subset\bigcup_{k=1}^r\overline{\mathcal V_k} \subset 
\bigcup_{k=1}^r\overline{\mathcal V_k}(\epsilon_0)\subset\mathcal G$
for some $\epsilon_0>0$, where 
$\mathcal V_k(\epsilon_0)=\mathcal V\{\varphi_j,\alpha_j-\epsilon_0\}_{j=1,\ldots,l}$.
Proposition \ref{upper0} gives
$$\limsup_{n\to\infty}\frac{1}{n}\log\xi_n(\overline{\mathcal V_k})\leq \sup_{\mathcal V_k(\epsilon)}F,$$
for every $\epsilon\in(0,\epsilon_0]$.
Hence
\begin{align*}
 \limsup_{n\to\infty}\frac{1}{n}
\log\xi_n\left(\bigcup_{k=1}^r\overline{\mathcal V_k} \right)
&\leq
\max_{1\leq k\leq r} \limsup_{n\to\infty}\frac{1}{n}\log
\xi_n(\overline{\mathcal V_k})\\
& \le \max_{1\leq k\leq  r}  
\sup_{ \mathcal V_k (\epsilon)}F  
\\ & \le
\sup_{\mathcal G}F,
\end{align*}
and thus
$
\displaystyle{\limsup}$$(1/n)\log\xi_n(\mathcal K)
\le \sup_{\mathcal G}F.$
Since $\mathcal G$ is an arbitrary open set containing 
 $\mathcal K$ it follows that
\begin{align*}
\limsup_{n\to\infty}\frac{1}{n}\log\xi_n(\mathcal K)
&\leq
\inf_{\mathcal G \supset \mathcal{K}} \sup_{\mathcal G}F=
\inf_{\mathcal G \supset \mathcal{K}} \sup_{\mathcal G} (-I)=
-\inf_{\mathcal K}I.\end{align*}
The last equality is due to the upper semi-continuity of $-I$. 
In the case $X$ is finitely primitive, 
the upper bounds for all closed sets and for all sequences of distributions
other than $\{\xi_n\}$ follow from
 Propositions \ref{tight} and
 \ref{upper0} in the same way.

The lower bound 
\eqref{low} for all open sets obtained in $\S3$ and the exponential tightness in Proposition \ref{tight} together imply that 
the rate function
$I$ is the good rate function \cite[Lemma 1.2.18(b)]{DemZei98}.
The convexity of $I$ follows from the affine character of $F$. The proofs of Theorems A, B
and C are now complete.
\end{proof}


The following example for the full shift is due to Jenkinson-Mauldin-Urba\'nski \cite[p.774]{JMU14}.
  For each integer $k\geq1$ denote by $\mu_k$ the Bernoulli measure generated by the collection $[k]$, $[k+1],\ldots,[k+2^k-1]$
 of $1$-cylinders.  Then $h(\mu_k)=k\log2$ holds.
 Put $\nu_k=(1-1/k)\delta_{\bar 0}+(1/k)\mu_k$ where $\bar 0=000\cdots$.
 Then $h(\nu_k)=\log2$, $h(\delta_{\bar 0})=0$ and $\nu_k$ converges to $\delta_{\bar 0}$ in the weak*-topology
 as $k\to\infty$.
 
 One can replace $\delta_{\bar 0}$ by an arbitrary measure with finite entropy and repeat
 the same construction to show that the entropy is not upper semi-continuous at this measure.
 As a result, the function $F$ in Theorem A is not upper semi-continuous at every measure with finite entropy.

\subsection*{Acknowledgments}
 I thank 
 Naotaka Kajino, Makiko Sasada and Mike Todd for fruitful discussions.
I also thank the referee for his or her careful reading of the manuscript and giving useful comments.
This research was partially supported by
the Grant-in-Aid for Young Scientists (A) of the JSPS 15H05435,
 the Grant-in-Aid for Scientific Research (B) of the JSPS 16KT0021 and
 the JSPS Core-to-Core Program ``Foundation
of a Global Research Cooperative Center in Mathematics focused on Number Theory and Geometry''.

\end{document}